\documentclass[a4paper,11pt]{article}
\pdfoutput=1
\usepackage{authblk}
\usepackage{a4wide,amssymb, amsmath, amsthm}
\usepackage{t1enc,lmodern}
\usepackage[utf8]{inputenc}
\usepackage[english]{babel}
\usepackage{hyperref}
\usepackage{natbib}
\usepackage{caption}
\usepackage[linesnumbered,lined,boxed]{algorithm2e}

\emergencystretch=\hsize
\def\cc{{\mathcal C}}

\def\cf{{\mathcal F}}
\def\cg{{\mathcal G}}
\def\ch{{\mathcal H}}
\def\ci{{\mathcal I}}

\def\ct{{\mathcal T}}

\def\D{{\mathbb D}}
\def\E{{\mathbb E}}
\def\I{{\mathbb I}}
\def\L{{\mathbb L}}
\def\N{{\mathbb N}}
\def\P{{\mathbb P}}

\def\R{{\mathbb R}}

\def\ee{{\bf e}}
\def\hh{\widehat{H}}
\def\l{\ell}
\def\ltwo{L^2([0,T])}
\def\s{\star}

\def\ind#1{{\bf 1}_{#1}}
\def\cond{\; | \;}

\def\abs#1{\left|#1\right|}

\def\Var{\mathop{\rm Var}\nolimits}
\def\norm#1{\mathop{\left\| #1 \right\|}\nolimits}
\def\inv#1{\mathop{\frac{1}{ #1}}\nolimits}
\def\expp#1{\mathop {\mathrm{e}^{ #1}}}
\def\esssup{\mathop {\mathrm{ess sup}}\nolimits}

\def\spaceIndex{A^{\otimes d}_{p,n}}
\theoremstyle{plain}
\newtheorem{thm}{Theorem}[section]

\newtheorem{prop}[thm]{Proposition}

\newtheorem{cor}[thm]{Corollary}
\theoremstyle{definition}
\newtheorem{rem}[thm]{Remark}

\title{Pricing American options using martingale bases}
\date{\today}
\author{Jérôme Lelong\footnote{This project was supported by the Finance for Energy Market
Research Centre, www.fime-lab.org. \\
The High Performance Computations presented in this paper were performed using the Froggy
platform of the CIMENT infrastructure (https://ciment.ujf-grenoble.fr), which is supported
by the Rhône-Alpes region (GRANT CPER07\_13 CIRA) and the Equip@Meso project (reference
ANR-10-EQPX-29-01) of the programme Investissements d'Avenir supervised by the Agence
Nationale pour la Recherche.} \\
\vskip1em \textit{Laboratoire Jean Kuntzmann, Université
    Grenoble Alpes, FRANCE.} \\ \vskip1em
  \texttt{jerome.lelong@imag.fr}}

\begin{document}
\maketitle

\begin{abstract}
  In this work, we propose an algorithm to price American options by directly solving the
  dual minimization problem introduced by \cite{rogers}. Our approach relies on approximating the
  set of uniformly square integrable martingales by a finite dimensional Wiener chaos
  expansion. Then, we use a sample average approximation technique to efficiently solve
  the optimization problem. Unlike all the regression based methods, our method can
  transparently deal with path dependent options without extra computations and a
  parallel implementation writes easily with very little communication and no centralized
  work. We test our approach on several multi--dimensional options with up to $40$ assets
  and show the impressive scalability of the parallel implementation. \\


\end{abstract}

\section{Introduction}
\label{sec:intro}

The pricing of American options quickly becomes challenging as the dimension increases and
the payoff gets complex. Many people have contributed to this problem usually by
considering its dynamic programming principle formulation \cite{tilley93},
\cite{carriere96}, \cite{tsit:vanr:01}, \cite{LS01}, \cite{brgl04} and \cite{bapa03}.
Among this so extensive literature, the practitioners seem to prefer the iterative optimal
policy approach proposed by \cite{LS01}, which proves to be quite efficient in many
situations. However, \emph{true} path--dependent options cannot be handled by this
approach. Solving the dynamic programming principle requires the computation of a
conditional expectation, which is eventually dealt with regression techniques. These
techniques are know to suffer from the curse of dimensionality: global regression methods
lead to high dimensional linear algebra problems, whereas local methods see the number of
domains blow up with the dimension. Despite the numerous parallel implementation of this
techniques (see for instance \cite{dbggs_10,CPE:CPE2862}), we cannot expect to obtain a
fully scalable algorithm. In this work, we follow the dual approach initiated by
\cite{rogers}, and \cite{da:ka:94}, which can naturally handle path dependent options. To
make it implementable, we need a smart and finite dimensional approximation of the set of
uniformly integrable martingales. We chose the set of truncated Wiener chaos expansions,
which have some magic features in our problem: it regularizes the optimization problem and
computing its conditional expectation exactly is straightforward. Then, the
pricing problem boils down finite dimensional, convex and differentiable optimization
problem.  The optimization problem is solved using a \emph{Sample Average Approximation}
(see \cite{MR1241645}), which can be easily and efficiently implemented using parallel
computing.  \\

We fix some finite time horizon $T>0$ and a filtered probability space $(\Omega, \cf,
(\cf_t)_{0 \le t \le T}, \P)$, where $(\cf_t)_{0 \le t \le T}$ is supposed to be the
natural augmented filtration of a $d-$dimensional Brownian motion $B$.  On this space, we
consider an adapted process $(S_t)_{0 \le t \le T}$ with values in $\R^{d'}$ modeling a
$d'$--dimensional underlying asset. The number of assets $d'$ can be smaller than the
dimension $d$ of the Brownian motion to encompass the case of stochastic volatility models
or stochastic interest rate.  We assume that the short interest rate is modelled by an
adapted process $(r_t)_{0 \le t \le T}$ with values in $\R_+$ and that $\P$ is an
associated risk neutral measure.  We consider an adapted payoff process $\tilde Z$ and
introduce its discounted value process $\left(Z_t = \expp{-\int_0^t r_s ds} \tilde
Z_t\right)_{0 \le t \le T}$. We assume that the paths of $Z$ are right continuous and that
$\sup_{t \in [0,T]} \abs{Z_t} \in L^2$. The process $\tilde Z$ can obviously take the
simple form $(\phi(S_t))_{t \le T}$ but it can also depend on the whole path of $S$ up to
the current time. So, our framework transparently deals with path--dependent option, which
are far more difficult to handle using regression techniques.

We consider the American option paying $\tilde Z_t$ to its holder if exercised at time
$t$.  Standard arbitrage pricing theory defines the discounted time-$t$ value of the
American option to be
\begin{equation}
  \label{eq:price_t_amer}
  U_{t} = \esssup_{\tau \in \ct_{t}} \E[Z_\tau | \cf_{t_k}]
\end{equation}
where $\ct_t$ denotes the set of $\cf-$stopping times with values in $[t,T]$. The
integrability properties of $Z$ ensure that $U$ is a supermartingale of class (D) and
hence has a Doob--Meyer decomposition
\begin{equation}
  \label{eq:doob-meyer_amer}
  U_{t} = U_0 + M^\s_{t} - A^\s_{t}
\end{equation}
where $M^\s$ is a martingale vanishing at zero and $A^\s$ is a predictable integrable
increasing process also vanishing at zero. With our assumptions on $Z$, $M^\s$ is square
integrable. \cite{rogers} found an alternative representation of the
price at time-$0$ of the American option as the minimum value of the following
optimization problem
\begin{equation}
  \label{eq:prix_dual}
  U_0 = \inf_{M \in H_0^2} \E\left[ \sup_{t \le T} (Z_{t} - M_{t}) \right]
  = \E\left[ \sup_{t \le T} (Z_{t} - M^\s_{t}) \right]
\end{equation}
where $H_0^2$ denotes the set of square integrable martingales vanishing at zero. A
martingale reaching the infimum is called an \emph{optimal} martingale. As the dual price
problem writes as a convex minimisation problem, the set of all optimal martingales is a
convex subset of $H^2_0$. Among the martingales reaching the infimum
in~\eqref{eq:prix_dual}, some of them actually satisfy the pathwise equality $\sup_{t \le
  T} Z_{t} - M_{t} = U_0$. These martingales are called \emph{surely optimal}.  Any surely
optimal martingale reaches the lower bound in~\eqref{eq:prix_dual} but not all optimal
martingales are surely optimal. We refer to \cite{schoen12-1} for a detailed
characterisation of optimal martingales.  Anyway, \cite{jamshidian07} proved the
uniqueness of surely optimal martingales within the continuing region, ie. for any surely
optimal martingale $M$  and any optimal strategy $\tau$, $(M_{t \wedge \tau})_t = (M^\s_{t
  \wedge \tau})_t$ a.s. 

The most famous method using the dual representation~\eqref{eq:prix_dual} is probably the
primal--dual approach of~\cite{An04primal}, which heavily relies on the knowledge of an
optimal exercising policy. The a priori knowledge may take the form of nested Monte Carlo
simulations as in \cite{schoenmakers2005robust}, and \cite{kolodko2004upper}. To
circumvent this difficulty, \cite{MR2669406} explained how to construct a \emph{good}
martingale. In a Wiener framework, \cite{Bel09} investigated this approach by relying on
the martingale representation theorem to build \emph{good} martingales. When trying to
practically use the dual formulation~\eqref{eq:prix_dual}, the first difficulty is
to find a rich enough but finite dimensional approximation of $H_0^2$ and then we face a
finite although potentially high--dimensional minimization problem (see
\cite{belomestny13-2} for one way of handling this approach).

The minimization problem~\eqref{eq:prix_dual} can be equivalently formulated as
\begin{equation}
  \label{eq:prix_dual2} 
  U_0 = \inf_{X \in L^2_0(\Omega,\cf_T, \P)}
  \E\left[ \sup_{0 \le t \le T} (Z_{t} - \E[X|\cf_{t}] )
  \right]
\end{equation}
where $L^2_0(\Omega,\cf_T, \P)$ is the set of square integrable $\cf_T-$ random variables with
zero mean. In this work, we suggest to use the truncated Wiener chaos expansion as a finite
dimensional approximation of $L^2(\Omega, \cf_T, \P)$. Since Wiener chaos are orthogonal
for the $L^2$ inner product, the computations of the conditional expectations
$\E[X|\cf_{t}]$ become straightforward and boil down to dropping some terms in the chaos
expansion, which makes our approach very convenient. Based on this approximation, we
propose a scalable algorithm and study its convergence.

The paper starts with the presentation of the Wiener chaos expansion and some of its
useful properties in Section~\ref{sec:wiener}. Then, we can develop the core of our work
in Section~\ref{sec:wiener-saa} in which we explain how the price of the American option
can be approximated by the solution of a finite dimensional optimization problem. First,
we analyze the properties of the optimization problem in order to prove the convergence of
its solution to the American option price. Second, we study its sample average
approximation, which makes the problem tractable, and prove its convergence. Based on all
these theoretical results, we present our algorithm in Section~\ref{sec:algo} and discuss
its parallel implementation on distributed memory architectures. Finally, some numerical
examples are presented in Section~\ref{sec:numerics}.

\section*{Notation}

\begin{itemize}
  \item For $n \ge 1$, $0= t_0 < t_1 < \dots < t_n=T$ is a time grid of $[0,T]$ satisfying
    $\lim_{n \to \infty} \sup_{0 \le k \le n-1} \abs{t_{k+1} - t_k} = 0$.
  \item For $n \ge 1$, the discrete time filtration $\cg$ is defined by $\cg_k =
    \sigma(B_{t_{i+1}} - B_{t_i}, i=0,\dots,k-1)$ for all $1 \le k \le n$, while $\cg_0$
    is the trivial sigma algebra. Obviously, $\cg_k \subset
    \cf_{t_k}$ for all $0 \le k \le n$.
  \item For $1 \le q \le d$, $\I(r) \in \{0,1\}^{n}$ denotes the vector $(\underbrace{0,\dots, 0}_{r-1},
    1, \underbrace{0, \dots, 0}_{n-r})$.
  \item For $1 \le q \le d$, and $1 \le r \le n$, $\I(r, q) \in \N^{n \times d}$ with all
    components equal to $0$ except the component with index $(r, q)$ which is equal to
    $1$.
\end{itemize}

\section{Wiener chaos expansion}
\label{sec:wiener}

For the sake of clearness, we first present the Wiener chaos expansion in the
case $d=1$ (ie. $B$ is a real valued Brownian motion). 

\subsection{General framework in dimension one}

\paragraph{Iterated stochastic integral approach.}
For a sequence of deterministic functions $(f_n)_{n \ge 1}$ such that $h_n :
[0,T]^n \longrightarrow \R$ and $\int_{[0,T]^n} |f_n(t)|^2 dt< \infty$, we define the
iterated stochastic integral by
\begin{align}
  \label{eq:Jn}
  \begin{cases}
  V_1(f_1) &=  \int_0^T  f_1(t_1) dB_{t_1} \\
  V_n(f_n) &= \int_0^T  \int_0^{t_n} \dots \int_0^{t_2}
  f_n (t_n, \dots, t_1) \; dB_{t_1} \dots dB_{t_n}, \quad \forall \; n \ge 2
\end{cases}
\end{align}
The set $\{V_n(f) : f \in L^2([0,T]^n)\}$ is a subspace of $L^2(\Omega, \cf_T,
\P)$, whose closure is often referred to as the Wiener chaos of order $n$.

We know from \cite{nualart_98} that any square integrable, real valued
and $\cf_T-$ measurable random variable $F$ can be expanded as a series of
iterated stochastic integrals
\begin{align}
  \label{eq:chaos-expansion}
  F = \E[F] + \sum_{n \ge 1} V_n(f_n)
\end{align}
where the deterministic functions $(f_i)_{i \ge 1}$ are symmetric. We define the
chaos expansion of order $p \ge 1$ as
\begin{align}
  \label{eq:chaos-expansion-p}
  C_p(F) = \E[F] + \sum_{n = 1}^p V_n(f_n)
\end{align}
which corresponds to the truncation of the sum in
Equality~\eqref{eq:chaos-expansion} to $p$ terms. For instance, for $p=2$,
$C_2(F)$ only involves a Wiener integral and a double stochastic integral
apart from the constant term $\E[F]$.

\paragraph{Hermite polynomials approach.} The iterated stochastic integral
approach to the Wiener chaos expansion is not applicable in practice and
cannot be generalized to multi--dimensional Brownian motions.
Hopefully, this expansion can be formulated in terms of Hermite polynomials. 

Let $H_i$ be the $i-th$ Hermite polynomial defined by
\begin{align}
  \label{eq:hermite}
  H_0(x) = 1; \qquad 
  H_i(x) = (-1)^i \expp{x^2/2} \frac{d^i}{dx^i} (\expp{-x^2/2}),
  \mbox{ for } i \ge 1.
\end{align}
They satisfy for all integer $i$, $H_i' = H_{i-1}$ with the convention $H_{-1} = 0$.
We recall that if $(X,Y)$ is a random normal vector with $\E[X] = \E[Y] = 0$ and
$\E[X^2] = \E[Y^2] = 1$ 
\begin{equation}
  \label{eq:prop-H}
  \E[H_i(X) H_j(Y)]  = i!\left( \E[XY] \right)^i \; \ind{i = j}.
\end{equation}
For all $i \ge 0$, we define the spaces
\begin{align}
  \label{eq:Hspaces}
  \ch_i = \span \left\{ H_i\left(\int_0^T f_t dB_t\right)\; : \; f \in L^2([0,T]) \right\}
\end{align}
whose $L^2$ closure corresponds to the Wiener chaos of order $i$.

We consider the indicator functions of the grid defined by $t_0 < t_1 < \dots. < t_n$ 
\begin{align}
  \label{eq:gi}
  f_i (t) = \ind{]t_{i-1}, t_i]}(t) / \sqrt{t_{i} - t_{i-1}}, \quad i=1,\dots,n, \;
\end{align}
With this choice for the $(f_i)_i$, 
\begin{align*}
  \int_0^T f_i(t) dB_t = \frac{B_{t_i} - B_{t_{i-1}}}{\sqrt{t_{i} -  t_{i-1}}} = G_i.
\end{align*}
Note that the random variables $G_i$ are i.i.d. following the standard normal
distribution.  We complete these $n$ functions $f_1, \dots, f_n$ into an
orthonormal basis of $L^2([0,T])$ and introduce the truncated chaos expansion
of order $p$ on the basis of $L^2$
\begin{align}
  \label{eq:chaos-p}
  C_{p,n}(F) = \sum_{\alpha \in A_{p,n}} \lambda_\alpha \prod_{i \ge 1}
  H_{\alpha_i} (G_i)
\end{align}
where $A_{p,n} = \{ \alpha \in \N^n \; : \; \norm{\alpha}_1 \le p \}$ with
$\norm{\alpha}_1 = \sum_{i \ge 0} \alpha_i$.  Using Equation~\eqref{eq:prop-H}, we deduce
that the coefficients of the above decomposition are uniquely determined by
\begin{align}
  \label{eq:chaos-d}
  \lambda_\alpha = \frac{\E \left[
      F \prod_{i \ge 1} H_{\alpha_i} (G_i) \right]}{\left(\prod_{i \ge 1}
      \alpha_i! \right)}.
\end{align}
This formula can be rewritten more clearly by introducing the generalized
Hermite polynomials defined for any multi--index $\alpha = (\alpha_i)_{i \ge 1}
\in \N^\N$
\begin{align}
  \label{eq:Hnd}
  \hh_\alpha (x) = \prod_{i \ge 1} H_{\alpha_i} (x_i), \quad \text{for } x \in
  \R^\N.
\end{align}
With this notation, Equation~\eqref{eq:chaos-p} becomes
\begin{align*}
  C_{p,n}(F) = \sum_{\alpha \in A_{p,n}} \lambda_\alpha \hh_{\alpha} (G_1, \dots, G_n).
\end{align*}
\begin{prop}
  \label{prop:Et-chaos}
  Let $F$ be a real valued random variable in $L^2(\Omega, \cf_T, \P)$ and let $k
  \in \{1, \dots, n\}$ and $p \ge 0$
  \begin{align*}
    \E[C_{p,n}(F) | \cf_{t_k}] = \sum_{\alpha \in A_{p,n}^k } \lambda_\alpha \;
    \hh_\alpha (G_1, \dots, G_{n})
  \end{align*}
  with $A_{p,n}^k = \{ \alpha \in \N^n \; : \; \norm{\alpha}_1 \le p, \; \alpha_\l = 0
    \; \forall \l > k \}$.
\end{prop}
\begin{proof}
  Taking the conditional expectation in Eq.~\eqref{eq:chaos-p} leads to
  \begin{align}
    \label{eq:Et-proof}
    \E[C_{p,n} (F) | \cf_{t_k}] = \sum_{\alpha \in A_{p,n}} \lambda_\alpha
    \prod_{i=1}^k H_{\alpha_i} (G_i) \E\left[ \prod_{i=k+1}^n H_{\alpha_i} (G_i)
    \Big|\cf_{t_k} \right].
  \end{align}
  Since the Brownian increments after time $t_k$ are independent of
  $\cf_{t_k}$ and are independent of one another , $\E\left[ \prod_{i=k+1}^n
  H_{\alpha_i} (G_i) |\cf_{t_k} \right] = \prod_{i=k+1}^n \E\left[ H_{\alpha_i}
  (G_i) \right]$, which is zero as soon as $\sum_{i=k+1}^n \alpha_i > 0$. Hence,
  the sum in Equation~\eqref{eq:Et-proof} is reduced to the sum over the set of
  multi--indices $\alpha \in A_{p,n}$ such that $\alpha_i = 0$ for all $i>k$,
  which is exactly the definition of the set $A_{p,n}^k$. 
\end{proof}
\begin{rem}
  \label{rem:Et-chaos}
  Since the sum appearing in $\E[C_{p,n}(F) | \cf_{t_k}]$ is reduced to a sum over the set
  of multi--indices $\alpha \in A_{p,n}^k$, it actually only
  depends on the first $k$ increments $(G_1, \dots, G_k)$. One can easily check
  that $\E[C_{p,n}(F) | \cf_{t_k}]$ is actually given by the chaos expansion of $F$
  on the first $k$ Brownian increments. Hence, computing a conditional
  expectation simply boils down to dropping term. While it may look like a naive way to
  proceed, it is indeed correct in our setting.
\end{rem}

\begin{prop}
  \label{prop:D-chaos}
 Let $F$ be a real valued random variable in $L^2(\Omega, \cf_T, \P)$ and let $k
 \in \{1, \dots, n\}$ and $p \ge 1$. For all $t \in ]t_{r-1}, t_{r}]$ with $1 \le r \le k$,
  \begin{align*}
    D_t \E[C_{p,n}(F) | \cf_{t_k}] = \frac{1}{\sqrt{h}} 
   \sum_{\alpha \in A_{p,n}^{k} , \; \alpha_r \ge
     1} \lambda_\alpha \; \hh_{\alpha - \I(r)} (G_1, \dots, G_{n})
  \end{align*}
  where $\alpha - \I(r) = (\alpha_1, \dots, \alpha_{r-1}, \alpha_r -1,
  \alpha_{r+1}, \dots, \alpha_n)$. 
\end{prop}

\begin{proof}
  From Proposition~\ref{prop:Et-chaos}, we know that for all $1 \le k \le n$
  \begin{align*}
    \E[C_{p,n}(F) | \cf_{t_k}] = \sum_{\alpha \in A_{p,n}^k} \lambda_\alpha \;
    \hh_\alpha (G_1, \dots, G_{n})
  \end{align*}
  Let $r \le k$ and $t \in ]t_{r-1}, t_{r}]$. The chain rule for the Malliavin
  derivative yields 
  \begin{align*}
    D_t \E[C_{p,n}(F) | \cf_{t_k}] & =  \sum_{\alpha \in A_{p,n}^k} \lambda_\alpha \;
    D_t \left (\prod_{i=1}^k H_{\alpha_i} (G_i) \right) \\
    D_t \left (\prod_{i=1}^k H_{\alpha_i} (G_i) \right)& =  
    \prod_{i=1, i \ne r}^k H_{\alpha_i} (G_i) H_{\alpha_r}'(G_r) \\
    & = \ind{\alpha_r \ge 1}  \prod_{i=1, i \ne r}^k H_{\alpha_i} (G_i) H_{\alpha_r-1}(G_r)\\
    & = \ind{\alpha_r \ge 1}  \hh_{\alpha - \I(r)} (G_1, \dots, G_n). \qedhere
  \end{align*}
\end{proof}

\subsection{Multi--dimensional chaos expansion}

In the previous section, we explained how a random variable measurable for a
sigma field generated by a one--dimensional Brownian motion could be
approximated by a finite sum of Hermite polynomials of Brownian increments.

In this section, we are back to our original multi--dimensional setting, as explained in
Section~\ref{sec:intro}. The process $B$ is a Brownian motion with values in $\R^d$.  The
key idea to extend the Hermite polynomial expansion to a higher dimensional setting is to
consider a tensor product of Hermite polynomials evaluated on a tensor basis of
$L^2([0,T], \R^d)$.

Consider the functions $(h_i)_i$ with values in $\R^d$ defined by
\begin{align*}
  h_i^j(t) = \frac{\ind{]t_{i-1}, t_i]}(t)}{\sqrt{h}} \ee_j, \; i=1,\dots,n, \; j=1,\dots,d
\end{align*}
where $(\ee_1,\dots, \ee_d)$ denotes the canonical basis of $\R^d$.
The $p-th$ order Wiener chaos $\cc_{p,n}$ is defined as the closure of 
\begin{align*}
  \left\{ \prod_{j=1}^d \hh_{\alpha^j}(G^{j}_1, \dots, G^j_n) \; : \;
    \alpha \in (\N^n)^d, \, \norm{\alpha}_1 \le p \right\}
\end{align*}
where $\norm{\alpha}_1 = \sum_{i=1}^n \sum_{j=1}^d \alpha_i^j$ and $G^j_i =
\frac{B^j_{t_i} - B^j_{t_{i-1}}}{\sqrt{h}}$. Using the independence of the Brownian
increments and the orthogonality of the Hermite polynomials, the chaos expansion of a
square integrable random variable $F$ is given by
\begin{align*}
  C_{p,n}(F) = \sum_{\alpha \in A^{\otimes d}_{p,n}} \lambda_\alpha 
  \hh_{\alpha}^{\otimes d} (G_1, \dots, G_n)
\end{align*}
where 
\begin{align}
  \label{eq:Ad}
  \hh^{\otimes d}_\alpha (G_1, \dots, G_n) & = \prod_{j=1}^d
  \hh_{\alpha^j} (G^j_1, \dots, G^j_n) \quad \forall \alpha \in (\N^n)^d \nonumber\\
  A^{\otimes d}_{p, n} & = \left\{ \alpha \in (\N^{n})^d\; : \norm{\alpha}_1 \le p
\right\}.
\end{align}
With an obvious abuse of notation, we write, for  $\lambda \in \R^{A^{\otimes d}_{p, n}}$,
\begin{align*}
  C_{p,n}(\lambda) = \sum_{\alpha \in A^{\otimes d}_{p,n}} \lambda_\alpha 
  \hh_{\alpha}^{\otimes d} (G_1, \dots, G_n).
\end{align*}
We also introduce the set of multi--indices truncated after time $t_k$
\begin{align}
  \label{eq:Adk} A^{\otimes d, k}_{p,n} = \left\{ \alpha \in  A^{\otimes d}_{p, n} \; :
    \; \forall j \in \{1,\dots, d\}, \, \forall \l > k, \; \alpha^j_\l = 0 \right\}.
\end{align}
We introduce the set $\cc_{p,n}$ defined by
\begin{align*}
  \cc_{p,n} = \left\{ F \in L^2(\Omega, \cf_T, P) \; : \; F = C_{p,n}(F) \; a.s. \right\}.
\end{align*}

We can easily deduce the multidimensional counterpart of Proposition~\ref{prop:Et-chaos}
\begin{prop}
  \label{prop:Et-chaos-d}
  Let $F$ be a real valued random variable in $L^2(\Omega, \cf_T, \P)$ and let $k
  \in \{1, \dots, n\}$ and $p \ge 0$
  \begin{align*}
    \E[C_{p,n}(F) | \cf_{t_k}] = \sum_{\alpha \in A^{\otimes d, k}_{p,n}} \lambda_\alpha \;
    \hh^{\otimes d}_\alpha (G_1, \dots, G_{n}).
  \end{align*}
\end{prop}
\begin{rem}
  \label{rem:chaos-measurability} The discrete time sequence $(\E[C_{p,n}(F) |
  \cf_{t_k}])_{0 \le k \le n}$ is of course adapted to the filtration $(\cf_{t_k})_k$ but
  also to the smaller filtration $(\cg_k)_k$. This property plays a crucial when
  approximating a random variable $F \in \L^2(\Omega, \cg_n, \P)$ as we know from
  \cite[Theorem 1.1.1]{nualart_98} that in such a case $\lim_{p \to \infty} C_{p,n}(F) =
  F$ in the $L^2-$sense. This result holds for the fixed value $n$. If $F$ were only
  $\cf_T-$measurable and  not $\cg_n-$measurable, we would need to impose that $F \in
  \D^{1,2}$ to obtain $\lim_{p \to \infty, n \to \infty} C_{p,n}(F) = F$. In this latter
  case, it is required to let $n$  go to infinity to recover $F$.
\end{rem}

\begin{prop}
  \label{prop:D-chaos-d}
 Let $F$ be a real valued random variable in $L^2(\Omega, \cf_T, \P)$ and let $k
 \in \{1, \dots, n\}$ and $p \ge 1$. For $t > t_k$, $D_t \E[C_{p,n}(F) | \cf_{t_k}] = 0$. 

 For all $t \in ]t_{r-1}, t_{r}]$ with $1 \le r \le k$, and $q=1,\dots,d$,
  \begin{align*}
    D^q_t \E[C_{p,n}(F) | \cf_{t_k}] = \frac{1}{\sqrt{h}} 
    \sum_{\alpha \in A_{p,n}^{\otimes d, k}, \alpha_r^q \ge 1}
    \lambda_\alpha \; \hh^{\otimes d}_{\alpha - \I(r,q)} (G_1, \dots, G_{n})
  \end{align*}
  where  $(\alpha - \I(r,q))^j_i = \alpha^j_i - \ind{j = q, i = r}$.
\end{prop}
In this multi-dimensional setting, the Malliavin derivative operator is actually a
gradient operator $D_t = (D^1_t, \dots, D^d_t)$.
\begin{proof}
  From Proposition~\ref{prop:Et-chaos-d}, we know that for all $1 \le k \le n$
  \begin{align*}
    \E[C_{p,n}(F) | \cf_{t_k}]  &= \sum_{\alpha \in A_{p,n}^{\otimes d, k}} \lambda_\alpha \;
    \prod_{j=1}^d \hh_{\alpha^j} (G^j_1, \dots, G^j_{n}).
  \end{align*}
  Let $r \le k$ and $t \in ]t_{r-1}, t_{r}]$. Let $1 \le q \le d$. The chain rule for the
  Malliavin derivative yields 
  \begin{align*}
    &D^q_t \E[C_{p,n}(F) | \cf_{t_k}] \\
    & = \sum_{\alpha \in
      A_{p,n}^{\otimes d, k}} \lambda_\alpha \;
    D^q_t \left(\prod_{j=1}^d \hh_{\alpha^j} (G^j_1, \dots, G^j_{n}) \right) \\
    & = \sum_{\alpha \in A_{p,n}^{\otimes d, k}} \lambda_\alpha \;
    \left(\prod_{j=1, j \ne q}^d \hh_{\alpha^j} (G^j_1, \dots, G^j_{n}) \right)
    D^q_t\left( \hh_{\alpha^q} (G^q_1, \dots, G^q_{n})\right) \\
    & = \inv{\sqrt{h}} \sum_{\alpha \in A_{p,n}^{\otimes d, k}} \lambda_\alpha \;
    \left(\prod_{j=1, j \ne q}^d \hh_{\alpha^j} (G^j_1, \dots, G^j_{n})\right)
    \hh_{\alpha^q - \I(r)} (G^q_1, \dots, G^q_{n})  \\
    & = \inv{\sqrt{h}} \sum_{\alpha \in A_{p,n}^{\otimes d, k}, \alpha^q_r \ge 1} \lambda_\alpha \;
    \hh_{\alpha - \I(r, q)} (G_1, \dots, G_{n}). \qedhere
  \end{align*}
\end{proof}
\begin{rem}\label{rem:Dzero}
  The conditional expectation preserves the nature of a chaos expansion.
  Similarly, the Malliavin derivative of a chaos expansion still writes as a
  chaos expansion and hence is a Hermite polynomial of Brownian increments.
  The roots of a non zero polynomial being a zero measure set and since the
  Brownian increments have a joined density, the Malliavin derivative of a chaos
  expansion is almost surely non zero as soon as one of the coefficients
  $\lambda_\alpha$ is non zero for $\alpha \in A_{p,n}^{\otimes d, k}$ such that
  $\alpha^j_r \ge 1$ for some $j \in \{1, \dots, d\}$.
\end{rem}

For $i, k \in \{1, \dots, n\}$, with $i<k$, we introduce the set $A_{p,n}^{\otimes d, i:k}$
defined as $A_{p,n}^{\otimes d, k} \setminus A_{p,n}^{\otimes d, i}$.
\begin{align}
  \label{eq:Aik}
  A_{p,n}^{\otimes d, i:k} 
  &= \Big\{ \alpha \in (\N^n)^d \;
    : \; \norm{\alpha}_1 \le p, \mbox{ and } \forall 1 \le j \le d, \;
    \forall \ell \notin \{i+1, \dots, k\}, \; \alpha^j_\ell = 0 \Big\}.
\end{align}

\section{Pricing American options using Wiener chaos expansion and sample average approximation}
\label{sec:wiener-saa}

In this section, we aim at approximating the dual price~\eqref{eq:prix_dual2} by a
tractable optimization problem. This involves two kinds of approximations: first,
approximate the space $L_0^2(\Omega,\cf_T, \P)$ by a finite dimensional vector space;
second, replace the expectation by a sample average approximation.

The dual price writes 
\begin{equation*}
  \inf_{X \in L_0^2(\Omega,\cf_T, \P)}
  \E\left[ \sup_{0 \le t \le T} (Z_t - \E[X|\cf_{t}] )
  \right].
\end{equation*}
In this optimization problem, we replace $X$ by its chaos expansion $C_{p,n}(X)$, which
has no constant term as $\E[X] = 0$ and we approximate the supremum by a discrete time
maximum. Then, we face a finite dimensional minimization problem to determine the optimal
solution with the subset $\cc_{p,n}$
\begin{equation}
  \label{eq:prix_dual_approx}
  \inf_{\lambda \in \R^{\spaceIndex}, \; \lambda_0 = 0}
  \E\left[ \max_{0 \le k \le n} (Z_{ t_k} - \E[C_{p,n}(\lambda)|\cf_{t_k}] )
  \right].
\end{equation}
In Section~\ref{sec:stoch}, we prove that this optimization problem is convex and has a
solution (see Proposition~\ref{prop:existence}) and converges to the price of the American
option (see Proposition~\ref{prop:cv}). Moreover, as the cost function is differentiable,
any minimizer is a zero of the gradient (see Proposition~\ref{prop:differentiability}),
which makes it easier to derive an algorithm.

To come up with a fully implementable algorithm, Section~\ref{sec:saa} presents the sample
average approximation of~\eqref{eq:prix_dual_approx}, which consists in replacing the
expectation by a Monte Carlo summation. We prove in Proposition~\ref{prop:cv_saa} that the
solution of the sample average approximation converges to the solution
of~\eqref{eq:prix_dual_approx} when the number of samples goes to infinity.

\subsection{A stochastic optimization approach}
\label{sec:stoch}

We fix $p \ge 1$ and  define the random functions $v_{p,n}(\cdot, \cdot; Z, G) :
\R^{A^{\otimes d}_{p,n}} \times \{0, \dots, n\}$  by
\begin{align*}
  v_{p,n} (\lambda, k; Z, G) & = 
  Z_{t_k} - \sum_{\alpha \in A^{\otimes d}_{p, n}} \lambda_\alpha
  \E\left[ \hh^{\otimes d}_{\alpha}
    \left( G_1, \dots, G_n \right) \Big| \cf_{t_k}\right],
\end{align*}
With the help of Proposition~\ref{prop:Et-chaos}, the random functions $v_{p,n}$
can be rewritten
\begin{align}
  \label{eq:cost-traj} v_{p,n} (\lambda, k, Z, G) & = Z_{t_k} - \sum_{\alpha \in
    A^{\otimes d, k}_{p,n}} \lambda_{\alpha} \hh^{\otimes d}_{\alpha} \left(
    G_1, \dots, G_n \right). 
\end{align}
We consider the cost function $V_{p, n} : \R^{A^{\otimes d}_{p,n}} \rightarrow \R$ defined by 
\begin{align}
  \label{eq:cost}
  V_{p,n} (\lambda) & = \E\left[ \max_{0 \le k \le n} v_{p,n} (\lambda, k; Z, G)
  \right]
\end{align}
and we approximate the solution of ~\eqref{eq:prix_dual2} by
\begin{equation}
  \label{eq:min_approx}
  \inf_{\lambda \in \R^{\spaceIndex}, \; \lambda_0 = 0}
  V_{p, n}(\lambda).
\end{equation}

We introduce the set of random indices for which the pathwise maximum is
attained
\begin{align}
  \label{eq:max-indices}
  \ci(\lambda, Z, G) = \left\{ 0 \le k \le n \; : \; v_{p,n} (\lambda, k; Z, G) = 
  \max_{\ell \le n} v_{p,n} (\lambda, \ell; Z, G) \right\}.
\end{align}

\begin{prop}
  \label{prop:existence}
  The minimization problem~\eqref{eq:min_approx} has at least one solution.
\end{prop}
\begin{proof}
  As the supremum of linear functions is convex, the random function $\lambda
  \longmapsto \max_{k \le n} v_{p,n} (\lambda, t_k, Z, G)$ is almost surely
  convex. The convexity of $V_{p,n}$ ensues from the linearity of the expectation.

  Let us prove that $V_{p, n}(\lambda) \to \infty$ when $\abs{\lambda} \to \infty$. Note
  that $V_{p, n}(\lambda) \ge \E\left[ (C_{p,n}(\lambda))_- \right] \ge \inv{2} \E\left[
    \abs{C_{p,n}(\lambda)} \right]$, where we have used that $\abs{x} = 2 x_- + x$ and
  $\E[C_{p,n}(\lambda)] = 0$.
  \begin{align*}
    \E\left[ \abs{C_{p,n}(\lambda)} \right] = \abs{\lambda} \E\left[ \abs{C_{p,n}(\lambda
        / \abs{\lambda})} \right] \ge \abs{\lambda} \inf_{\mu \in \R^{A^{\otimes d}_{p,n}},
      \abs{\mu} = 1}  \E\left[ \abs{C_{p,n}(\mu)} \right].
  \end{align*}
  By a standard continuity argument, the infimum is attained.  Moreover, it is strictly
  positive as otherwise there would exist $\mu \in \R^{A^{\otimes d}_{p,n}}$ with
  $\abs{\mu} = 1$ s.t.  $\E\left[ \abs{C_{p,n}(\mu)} \right] = 0$.  Using the
  orthogonality of the family $\left(H^{\otimes d}_{\alpha}\right)_{\alpha \in A^{\otimes
      d}_{p,n}}$, we would immediately deduce that $\mu = 0$. Hence, we show that $V_{p,
    n}(\lambda) \to \infty$ when $\abs{\lambda} \to \infty$. The growth at infinity of
  $V_{p, n}$ combined with its convexity yields the existence of a solution to the
  minimization problem~\eqref{eq:min_approx}.
\end{proof}

Proposition~\ref{prop:existence} ensures the existence of $\lambda_{p,n}^\sharp$
solving~\eqref{eq:min_approx}, ie.
\begin{align}
  \label{eq:sol-approx}
  V_{p, n}(\lambda_{p,n}^\sharp) = \inf_{\lambda \; s.t. \; \lambda_0 = 0} V_{p, n}(\lambda).
\end{align}
Moreover, $\nabla V_{p,n}(\lambda_{p,n}^\sharp) = 0$. This characterization of an optimal
solution will be of prime importance to practically devise an algorithm.

\begin{prop}
  \label{prop:cv}
  The solution of the minimization problem~\eqref{eq:min_approx},
  $V_{p,n}(\lambda_{p,n}^\sharp)$, converges to $U_0$ when both $p$ and $n$ go to
  infinity.
\end{prop}
\begin{proof}
  We introduce the truncated chaos expansion of $M^\s_T$ and denote its coefficients by
  $\lambda^\s_{p,n}$, ie. $C_{p,n}(M_T^\s) = C_{p,n}(\lambda^\s_{p,n})$. Clearly, $U_0 \le
  V_{p, n}(\lambda_{p,n}^\sharp) \le V_{p, n}(\lambda^\s_{p,n})$. Then, we obtain the
  following result
  \begin{align}
    \label{eq:error_bound}
    0 \le V_{p,n}(\lambda_{p,n}^\sharp)  - U_0  & \le V_{p,n}(\lambda_{p,n}^\s)  - U_0 \nonumber\\
    & = \E\left[ \max_{k} (Z_{t_k} - \E[C_{p,n}(\lambda_{p,n}^\s)
      | \cf_{t_k}]) - \max_{k} (Z_{t_k} - M^\s_{t_k}) \right] \nonumber\\
    & \le \E\left[ \max_{k} \abs{M^\s_{t_k} -
        \E[C_{p,n}(\lambda_{p,n}^\s) | \cf_{t_k}]} \right] \nonumber\\
    & \le \E\left[ \max_{k} \E\left[\abs{M^\s_{T} -
          C_{p,n}(\lambda_{p,n}^\s)} | \cf_{t_k}\right] \right] \nonumber\\
    & \le \sqrt{\E\left[ \max_{k} \E\left[\abs{M^\s_{T} -
            C_{p,n}(\lambda_{p,n}^\s)} | \cf_{t_k}\right]^2 \right]} \nonumber\\
    & \le 2 \norm{M^\s_{T} - C_{p,n}(M_T^\s)}_2
  \end{align}
  where the last upper--bound ensues from Doob's inequality. Note that this bound does not
  depend on $\lambda_{p,n}^\sharp$. The convergence result when $p, n$ go to infinity
  ensues from \cite[Lemma 2, Lemma 19]{BriandLabart14}.
\end{proof}

\begin{rem}
  Note that if in the series of inequalities~\eqref{eq:error_bound}, we had dropped the
  second one, we would have come up in the end with
  $0 \le V_{p,n}(\lambda_{p,n}^\sharp)  - U_0  \le 2 \norm{M^\s_{T} -
    C_{p,n}(\lambda_{p,n}^\sharp)}_2$, which, by definition of the chaos expansion, is
  larger than $2 \norm{M^\s_{T} - C_{p,n}(M_T^\s)}_2$.
\end{rem}

\begin{cor}
  Consider the Bermudean option with exercising dates $t_0,\dots, t_n$ and with discounted
  payoff $(Z_{t_k})_k$ assumed to be $\cg-$adapted. Then, $V_{p,n}(\lambda_{p,n}^\sharp)$
  converges to the price of the Bermudean option when $p$ goes to infinity.
\end{cor}
\begin{proof}
  Let $\hat U_k$ be the price at time$-t_k$ of the Bermudean option. The sequence $(\hat
  U_k)_k$ is a supermartingale admitting the Doob--Meyer decomposition $\hat U_k = \hat U_0 +
  \hat M_k - \hat A_k$ where $\hat M$ is a square integrable $(\cg_k)_k-$martingale and $\hat
  A$ a predictable increasing process for the filtration $\cg$.
  The price at time$-0$ of the Bermudean option also writes
  \begin{equation*}
    \hat U_0 = \inf_{ X \in L_0^2(\Omega,\cg_n, \P)}
    \E\left[ \max_{0 \le k \le n} (Z_{t_k} - \E[X|\cg_{k}] ) \right].
  \end{equation*}
  Clearly, $C_{p,n}(\lambda) \in L^2(\Omega,\cg_n, \P)$ for any $\lambda$ such that
  $\lambda_0 = 0$ and moreover $V_{p,n}(\lambda_{p,n}^\sharp) \ge \hat U_0$. Then, by
  reproducing the steps in~\eqref{eq:error_bound}, we get
  \begin{align*}
    0 \le V_{p,n}(\lambda_{p,n}^\sharp) - \hat U_0 \le 2 \norm{\hat M_n - C_{p,n}(\hat
      M_n)}_2.
  \end{align*}
  We deduce from Remark~\ref{rem:chaos-measurability} that this upper--bound goes to zero
  as $p$ tends to infinity.
\end{proof}
Most convex optimization algorithms mainly rely on the gradient of the cost function. We
end this section by proving that $V_{p,n}$ is almost everywhere differentiable, which
implies that $\nabla V_{p,n}(\lambda_{p,n}^\sharp) = 0$.
\begin{prop}
  \label{prop:differentiability} Let $p \ge 1$. Assume that 
  \begin{multline}
    \label{eq:Dt_density}
    \forall 1 \le r \le k \le n, \; \forall F \; \cf_{t_k}-\text{measurable}, \; F \in
    \cc_{p-1,n}, \; F \ne 0, \;
    \exists \; q \in \{1, \dots, d\} \mbox{ s.t. }  \\
    \quad \P\left(\forall t \in ]t_{r-1},t_r], \; D_t^q Z_{t_k} + F = 0
      \cond Z_{t_k} > 0\right) = 0.
  \end{multline}
  Then,  the function $V_{p,n}$ is differentiable at all points $\lambda \in
  \R^{A^{\otimes d}_{p, n}}$ with no zero component and its gradient $\nabla V_{p,n}$ is
  given by
  \begin{align*}
    \nabla V_{p,n}(\lambda) = \E\left[ \E\left[ \hh^{\otimes d} (G_1, \dots,
        G_n)  \cond \cf_{t_i}\right]_{|\{i\} = \ci(\lambda, Z, G) } \right].
  \end{align*}
\end{prop}
We refer the reader to section~\ref{sec:density_condition} for a detailed
discussion on which kinds of models and payoffs satisfy~\eqref{eq:Dt_density}.
\begin{proof}
  We already know that the function $V_{p,n}$ is convex.   Moreover, for all $Z$ and $G$,
  the function $\lambda \longmapsto \max_{k \le n} v_{p,n} (\lambda, k, Z, G)$ has a
  subdifferential given by
  \begin{align*}
   \left\{ \sum_{i \in \ci(\lambda, Z, G)}
     \beta_i \E[\hh^{\otimes d} (G_1, \dots, G_n)| \cf_{t_i}] \; : \; \beta_i \ge 0, \;
     \beta_i \; \cf_T-measurable \;
    \mbox{s.t. } \sum_{i \in \ci(\lambda, Z, G)} \beta_i = 1
    \right\}
  \end{align*}
  Then, the expression of the subdifferential $\partial V_{p,n}(\lambda)$ ensues
  from~\cite{MR0329725}. 
  \begin{align*}
    \partial V_{p,n}(\lambda) =  \left\{ \E\left[\sum_{i \in \ci(\lambda, Z, G)}
      \beta_i \E[\hh^{\otimes d} (G_1, \dots, G_n)| \cf_{t_i}] \right] \; : \; \beta_i \ge 0, \;
     \beta_i \; \cf_T-meas., \; \sum_{i} \beta_i = 1 \right\}.
  \end{align*}

  It is sufficient to prove for any $\lambda$ with no zero component, the set
  $\ci(\lambda, Z, G)$ is almost surely reduced to a single value as in this case the
  subdifferential $\partial V_{p,n}(\lambda)$ contains a unique element, which is then the
  gradient.

  By the equality
  \begin{align*}
    \left\{ \exists t_i \ne t_k \; ; \; v_{p,n} (\lambda, i; Z, G)  = v_{p,n}
    (\lambda, k; Z, G)  \right\} = \bigcup_{i <  k \le n} 
    \left\{ v_{p,n} (\lambda, i; Z, G)  = v_{p,n} (\lambda, k; Z, G)  \right\},
  \end{align*}
  it is sufficient to prove that for any $i < k \le n$, $\P( v_{p,n} (\lambda, i, Z, G)  =
  v_{p,n} (\lambda, k, Z, G) ) = 0$. Fix $i < k$ and set $X_\lambda = v_{p,n} (\lambda, k,
  Z, G)  - v_{p,n} (\lambda, i, Z, G)$.  According to \cite[Theorem 2.1.3]{nualart_98}, it
  is sufficient to prove that $\norm{DX_\lambda}_{\ltwo} > 0$ .a.s to ensure that
  $X_\lambda$ is absolutely continuous with respect to the Lebesgue measure on $\R$ and
  hence is almost surely non zero. Since
  $\norm{DX_\lambda}_{\ltwo}^2 = \int_0^T |D_t X_\lambda |^2 dt \ge \int_{t_i}^{t_k} |D_t
  X_\lambda |^2 dt \ge \int_{t_i}^{t_k} (D^q_t X_\lambda )^2 dt$ for any $q \in \{1,
    \dots, d\}$.

  For $t \in [0,T]$, and $1 \le q \le d$,  the Malliavin derivative of $X_\lambda$ is given by
  \begin{align*}
    D^q_t X_\lambda & = D^q_t (Z_{t_k} - Z_{t_{i}}) -
    D^q_t\left(\sum_{\alpha \in A_{p;n}^{\otimes d,k}} \lambda_{\alpha}
      \hh^{\otimes d}_{\alpha} \left( G_1, \dots, G_n \right) - \sum_{\alpha \in
        A_{p;n}^{\otimes d, i}} \lambda_{\alpha} \hh^{\otimes d}_{\alpha} \left(
        G_1, \dots, G_n \right)\right) \\
    & = D^q_t (Z_{t_k} - Z_{t_{i}}) -
    D^q_t \left(  \sum_{\alpha \in A_{p;n}^{\otimes d, i:k}} \lambda_{\alpha}
      \hh^{\otimes d}_{\alpha} \left( G_1, \dots, G_n \right) \right).
  \end{align*}
   Clearly, w.p.1. $D^q_t X_\lambda = 0$ for all $t > t_k$. Hence,
  \[
    \left\{ D^q_t X_\lambda = 0 \; \forall t \in [0,T] \, a.e. \right\} \subset
   \bigcap_{i < r \le k} \left\{ D^q_t X_\lambda = 0 \; \forall t \in [t_{r-1},t_r] \, a.e. \right\}.
  \]
  From Proposition~\ref{prop:D-chaos}, we can deduce that for $i < r \le k$, and $t \in
  ]t_{r-1}, t_r]$
  \begin{align*}
    D^q_t X_\lambda = D^q_t (Z_{t_k})
    + \frac{1}{\sqrt{h}} \sum_{\alpha \in
      A_{p;n}^{\otimes d, i:k}, \alpha^q_r \ge 1} \lambda_{\alpha} \hh^{\otimes d}_{\alpha
      - \I(r, q)} \left( G_1, \dots, G_n \right). 
  \end{align*}
  Using the locality of the operator $D$, we know that a.s $D^q_t (\phi(S_{t_k})) = 0$ for
  all $t \in  ]t_{r-1}, t_r]$ on the set $\{\phi(S_{t_k}) = 0\}$. Hence, we can write
  \begin{align*}
     \P\left( \forall t \in  ]t_{r-1}, t_r], \; D_t X_\lambda = 0 \right) & =
    \P\left( \frac{1}{\sqrt{h}} \sum_{\alpha \in
      A_{p;n}^{\otimes d, i:k}, \alpha^q_r \ge 1} \lambda_{\alpha} \hh^{\otimes d}_{\alpha
      - \I(r, q)} \left( G_1, \dots, G_n \right) = 0, \; \phi(S_{t_k}) = 0
    \right)  \\
    & \quad + 
   \P\left( \forall t \in  ]t_{r-1}, t_r], \; D_t X_\lambda = 0 \cond Z_{t_k} > 0
    \right)  \P(Z_{t_k} > 0).
  \end{align*}
  As all the components of $\lambda$ are non zero,  $\frac{1}{\sqrt{h}} \sum_{\alpha \in
    A_{p;n}^{\otimes d, i:k}, \alpha^q_r \ge 1} \lambda_{\alpha} \hh^{\otimes d}_{\alpha
    - \I(r, q)} \left( G_1, \dots, G_n \right)$ is either a non zero constant if $p=1$
  or it has an absolutely continuous density thanks to Remark~\ref{rem:Dzero}. In both
  cases, it is a non zero element of $\cc_{p-1,n}$ and
  \[
    \P\left( \frac{1}{\sqrt{h}} \sum_{\alpha \in
        A_{p;n}^{\otimes d, i:k}, \alpha^q_r \ge 1} \lambda_{\alpha} \hh^{\otimes d}_{\alpha
        - \I(r, q)} \left( G_1, \dots, G_n \right) = 0  \right)  = 0.
  \]
  To treat the other term, we pick a $q \in \{1, \dots, d\}$ as in the assumption of the
  proposition (see \eqref{eq:Dt_density}) and it yields that
  \[
    \P\left( \forall t \in  ]t_{r-1}, t_r], \; D_t^q X_\lambda = 0 \cond Z_{t_k} > 0
    \right) = 0.
  \]
  Hence, we deduce that 
  \[
    \norm{DX_\lambda}_{\ltwo}^2 \ge \int_{t_{r-1}}^{t_r} (D^q_t X_\lambda )^2 dt > 0 \; a.s.
  \]
  which concludes the proof.
\end{proof}

\subsection{The Sample Average Approximation point of view}
\label{sec:saa}

From~\eqref{eq:error_bound}, we can approximate $U_0$ by solving the
minimization problem~\eqref{eq:min_approx}, which admits at least one solution
$\lambda_{p,n}^\sharp$, ie.
\begin{align*}
  V_{p, n}(\lambda_{p,n}^\sharp) = \inf_{\lambda \in \spaceIndex, \; \lambda_0 = 0} V_{p, n}(\lambda)
\end{align*}
where $V_{p,n}$ defined by~\eqref{eq:cost} is an expectation, which is barely tractable.
To practically solve such a problem, two differently approaches are commonly used. Either,
one uses a stochastic algorithm or one replaces the expectation by a sample average
approximation. In this work, we target large problems, which puts scalability as a primary
requirement. The intrinsic sequential nature of stochastic algorithms has led us to
prefer the sample average approximation approach. Moreover, we are more interested in the
value function at the minimum rather than in its minimizer and unlike stochastic
algorithm, standard optimization algorithms provide both at once.\\

We introduce the sample average approximation of $V_{p,n}$ defined by 
\begin{align*}
  V_{p,n}^m(\lambda) = \inv{m} \sum_{i=1}^m \max_{0 \le k \le n} v_{p,n}(\lambda, k;
  Z^{(i)}, G^{(i)})
\end{align*}
where $(Z^{(i)}, G^{(i)})_{1 \le i \le m}$ are i.i.d samples from the distribution of $(Z,
G)$. 

For large enough $m$, $V^m_{p,n}$ inherits from the smoothness of $V_{p,n}$ and is in
particular convex and a.s. differentiable at any point with no zero component. Then, we
easily deduce from Proposition~\ref{prop:existence} that there exits $\lambda^m_{p,n}$
such that
\begin{align*}
  V_{p,n}^m(\lambda^m_{p,n}) = \inf_{\lambda \in \R^{\spaceIndex}, \; \lambda_0 = 0} V_{p,n}^m(\lambda)  
\end{align*}
and moreover $\nabla V^m_{p, n}(\lambda^m_{p,n}) = 0$.  The main difficulty in studying
the convergence of $V_{p,n}^m(\lambda^m_{p,n})$ when $m$ goes to infinity comes from the
non compactness of the set $\R^{\spaceIndex}$. To circumvent this problem, we adapt  to
non strictly convex problems the technique used in \cite{jour:lelo:09}.

\begin{prop}
  \label{prop:cv_saa} The sequence $V_{p,n}^m(\lambda^m_{p,n})$ converges a.s. to
  $V_{p,n}(\lambda_{p,n}^\sharp)$ when $m \to \infty$. Moreover, the distance between
  $\lambda^m_{p,n}$ and the convex set of minimizers in~\eqref{eq:min_approx} converges to
  zero as $m$ goes to infinity.
\end{prop}
\begin{proof}
  The random function $\lambda \in \R^{\spaceIndex} \mapsto \max_{0 \le k \le n}
  v_{p,n}(\lambda, k; Z, G) $ is a.s. continuous. For $\Lambda > 0$,
  \begin{align*}
     \sup_{\abs{\lambda} \le \Lambda} & \max_{0 \le k \le n}
    v_{p,n}(\lambda, k; Z, G)  \\
    & \le \max_{0 \le k \le n} Z_{t_k} + \sup_{\abs{\lambda} \le \Lambda} \max_{0 \le k \le
      n}  \sum_{\alpha \in A^{\otimes d}_{p, n}} \lambda_\alpha \E\left[ \hh^{\otimes
        d}_{\alpha} \left( G_1, \dots, G_n \right) \Big| \cf_{t_k}\right] \\
    & \le \max_{0 \le k \le n} Z_{t_k} + \Lambda \sup_{\abs{\lambda} = 1} \max_{0 \le k \le
      n}  \sum_{\alpha \in A^{\otimes d}_{p, n}} \lambda_\alpha \E\left[ \hh^{\otimes
        d}_{\alpha} \left( G_1, \dots, G_n \right) \Big| \cf_{t_k}\right] \\
    & \le \max_{0 \le k \le n} Z_{t_k} + \Lambda \max_{0 \le k \le n}  \sum_{\alpha \in
      A^{\otimes d}_{p, n}} \abs{\E\left[ \hh^{\otimes d}_{\alpha} \left( G_1, \dots, G_n
        \right) \Big| \cf_{t_k}\right]} \\
    & \le \max_{0 \le k \le n} Z_{t_k} + \Lambda 
   \sum_{\alpha \in A^{\otimes d}_{p, n}}  \max_{0 \le k \le n}  
    \E\left[\abs{ \hh^{\otimes d}_{\alpha} \left( G_1, \dots, G_n
        \right) }\Big| \cf_{t_k}\right] \\
    & \le \max_{0 \le k \le n} Z_{t_k} + \Lambda 
    \sum_{\alpha \in A^{\otimes d}_{p, n}}  \sum_{k = 0}^n
    \E\left[\abs{ \hh^{\otimes d}_{\alpha} \left( G_1, \dots, G_n
        \right) }\Big| \cf_{t_k}\right].
  \end{align*}
  The right hand side of the above inequality is integrable. We apply \cite[Lemma A1 
  Chapter 2]{MR1241645} to deduce that a.s. $V_{p,n}^m$ converges locally uniformly to
  $V_{p,n}$.  From the proof of the Proposition~\ref{prop:existence}, there exits $\Lambda
  > 0$ such that
  \begin{align*}
    \gamma = \inf_{\abs{\lambda - \lambda_{p,n}^\sharp} \ge
      \Lambda} V_{p,n}(\lambda) - V_{p,n}(\lambda_{p,n}^\sharp) > 0.
  \end{align*}
  The local uniform convergence of $V_{p,n}^m$ to $V_{p,n}$ ensures that
  \begin{align*}
    \exists \; m_\gamma \in \N^*, \; \forall m \ge m_\gamma, \; \forall \lambda \mbox{
      s.t. } \abs{\lambda - \lambda_{p,n}^\sharp} \le \Lambda, \quad
      \abs{V_{p,n}^m(\lambda) - V_{p,n}(\lambda)} \le \frac{\gamma}{3}.
  \end{align*}
  For $m \ge m_\gamma$ and $\lambda$ such that $\abs{\lambda - \lambda_{p,n}^\sharp} \ge
  \Lambda$, we deduce, using the convexity of $V_{p,n}^m$, that
  \begin{align*}
    V_{p,n}^m(\lambda) - &  V_{p,n}^m(\lambda_{p,n}^\sharp) \\
    &\ge \frac{\abs{\lambda - \lambda_{p,n}^\sharp}}{\Lambda} \left\{ V_{p,n}^m\left(
          \lambda_{p,n}^\sharp + \Lambda \frac{\lambda -
            \lambda_{p,n}^\sharp}{\abs{\lambda - \lambda_{p,n}^\sharp}} \right) -
        V_{p,n}^m(\lambda_{p,n}^\sharp) \right\}\\
    &\ge \frac{\abs{\lambda - \lambda_{p,n}^\sharp}}{\Lambda} \left\{ V_{p,n}\left(
          \lambda_{p,n}^\sharp + \Lambda \frac{\lambda -
            \lambda_{p,n}^\sharp}{\abs{\lambda - \lambda_{p,n}^\sharp}} \right) -
        V_{p,n}(\lambda_{p,n}^\sharp) - \frac{2 \gamma}{3}\right\} \ge
      \frac{\gamma}{3}.
  \end{align*}
  Since $V_{p,n}^m(\lambda^m_{p,n}) - V_{p,n}^m(\lambda_{p,n}^\sharp) \le 0$, we conclude
  that the above inequality does not hold for $\lambda_{p,n}^m$, which proves that
  $\abs{\lambda^m_{p,n} - \lambda_{p,n}^\sharp} < \Lambda$ for $m \ge m_\gamma$.

  Hence, for $m \ge m_\gamma$,  it is sufficient to minimize $V_{p,n}^m$ on the compact
  set $\{\lambda \; : \; \abs{\lambda - \lambda_{p,n}^\sharp} \le \Lambda\}$. Now, we can
  apply \cite[Theorem A1 of Chapter 2]{MR1241645} to prove that
  $V_{p,n}^m(\lambda^m_{p,n})$ converges to $V_{p,n}(\lambda_{p,n}^\sharp)$ a.s. when $m$
  goes to infinity. The second assertion of our proposition is discussed right after the
  proof of Theorem A1 in \cite{MR1241645}.
\end{proof}

Although $V_{p,n}^{m}$ is not twice differentiable and the classical central limit theorem
for sample average approximations cannot be applied, we can study the variance of
$V_{p,n}^{m}(\lambda_{p,n}^m)$ and obtain some asymptotic bounds. Before stating our
result, we introduce, for $\lambda \in \R^{\spaceIndex}$, the notation $M_k(\lambda) =
\E[C_{p,n}(\lambda) | \cf_{t_k}]$ for $0 \le k \le n$. We write $M^{(i)}_k(\lambda)$ for
the value computed using the sample $G^{(i)}$.

\begin{prop}
  \label{prop:var_saa}
  Assume $\lambda^\sharp_{p,n}$ is unique. Then, 
  \begin{align*}
    \inv{m} \sum_{i=1}^m \left(\max_{0 \le k \le n} Z^{(i)}_{t_k} -
    M^{(i)}_k(\lambda^m_{p,n})\right)^2 - V^{m}_{p,n}(\lambda^m_{p,n})^2
  \end{align*}
  is a convergent estimator of $\Var(\max_{k \le 0 \le n} Z_{t_k} - M_k(\lambda_{p,n}^\sharp))$
  and moreover if $\lambda^m_{p,n}$ is bounded, $\lim_{m \to \infty} m \Var\left(V^{m}_{p,n}(\lambda^m_{p,n})\right) =
  \Var(\max_{k \le 0 \le n} Z_{t_k} - M_k(\lambda_{p,n}^\sharp))$. 
\end{prop}
\begin{proof}
  We know that $V^m_{p,n}(\lambda^m_{p,n})$ converges a.s. to
  $V_{p,n}(\lambda^\sharp_{p,n})$.  Following the beginning of the proof of
  Proposition~\ref{prop:cv_saa}, one can easily prove that a.s. the sequence of random
  functions $\zeta^m : \lambda \mapsto \zeta^m(\lambda) = \inv{m} \sum_{i=1}^m
  \left(\max_{0 \le k \le n} Z^{(i)}_{t_k} - M^{(i)}_k(\lambda)\right)^2$ converges
  locally uniformly to the function $\lambda \mapsto \E[(\max_{0 \le k \le n} Z_{t_k} -
  M_k(\lambda))^2]$. We have already seen that for large enough $m$, we can assume to have
  solved the optimization problem under a compact constraint.
  Hence, we deduce that $\inv{m} \sum_{i=1}^m \left(\max_{0 \le k \le n} Z^{(i)}_{t_k} -
  M^{(i)}_k(\lambda^m_{p,n})\right)^2 $ converges a.s. to $\E[(\max_{0 \le k \le n}
  Z_{t_k} - M_k(\lambda^\sharp_{p,n}))^2]$.  This proves the first statement of the
  proposition.

  As $\Var\left( V^m_{p,n}(\lambda^\sharp_{p,n}) \right) = m^{-1}\Var(\max_{k \le 0 \le n}
  Z_{t_k} - M_k(\lambda_{p,n}^\sharp))$, it is sufficient to compute
  \begin{align*}
    & \E\left[ \left( V_{p,n}^m(\lambda^m_{p,n}) - V^m_{p,n}(\lambda^\sharp_{p,n})
    \right)^2 \right] \\
    & \le  \inv{m} \sum_{i=1}^m \E\left[ \max_k \abs{M^{(i)}_k(\lambda^m_{p,n}) -
    M^{(i)}_k(\lambda^\sharp_{p,n})}^2 \right] \\
    & \le  \inv{m} \sum_{i=1}^m \E\left[ \abs{\lambda^m_{p,n} -
      \lambda^\sharp_{p,n}}^2 \max_k \abs{\E\left[ \hh^{\otimes d} (G^{(i)}_1, \dots,
      G^{(i)}_n)  \cond \cf_{t_k}\right]}^2\right]  \\
      & \le \frac{16}{9} \E\left[ \abs{\lambda^m_{p,n} -
      \lambda^\sharp_{p,n}}^4\right]^{1/2} \E\left[ \abs{ \hh^{\otimes d} (G^{(i)}_1, \dots,
      G^{(i)}_n) }^4\right]^{1/2} 
  \end{align*}
  where we have used Cauchy Schwartz' inequality and Doob's maximal inequality. Then, we
  easily conclude that $V^m_{p,n}(\lambda^m_{p,n}) - V^m_{p,n}(\lambda^\sharp_{p,n})$
  converges to $0$ in $\L^2$ if $\lambda^m_{p,n}$ is bounded. Hence, $\lim_{m \to \infty}
  \Var\left( V^m_{p,n}(\lambda^m_{p,n})\right) - \Var\left(
  V^m_{p,n}(\lambda^\sharp_{p,n})\right) = 0$. 
\end{proof}

Proposition~\ref{prop:var_saa} enables us to monitor the variance of our estimator online
as for a standard Monte Carlo estimator. Even though the terms involved in
$V^m_{p,n}(\lambda^m_{p,n})$  are not independent, the classical variance estimator gives
the right result. In practice, one should not feel concerned with the boundedness
condition used in the proposition as we know from the proof of
Proposition~\ref{prop:cv_saa} that for large enough $m$ we can impose a compactness
constraint to the optimization problem without changing its result. Hence, one can
pragmatically rely on the proposed variance estimator.

\section{The algorithm}
\label{sec:algo}

Any optimization algorithm requires to repeatedly compute $V^m_{p,n}$ and therefore the
truncated chaos expansion, which becomes the most time consuming part of our approach as
the dimension and/or $p$ increase. A lot of computational time can be saved by considering
slightly modified martingales, which only start the first time the option goes in the
money.

\subsection{An improved set of martingales}

We define the first time the option goes in the money by 
\[
  \tau_0 = \inf\{k \ge 0 \; : \; Z_{t_k} > 0\} \wedge n,
\]
which is a $\cf-$ stopping time and becomes a $\cg-$ stopping time when the sequence
$(Z_{t_k})_k$ is $\cg-$ adapted.   To consider martingales only starting once the option
has been in the money, we define
\begin{align*}
  N_k(\lambda) = \sum_{\ell = 1}^k (M_{\ell}(\lambda) - M_{\ell - 1}(\lambda)) 
  \ind{\ell - 1 \ge \tau_0} = (M_k(\lambda) - M_{\tau_0}(\lambda)) \ind{k > \tau_0} =
  M_k(\lambda) - M_{k \wedge \tau_0}(\lambda)
\end{align*}
We easily check that $N(\lambda)$ is a $(\cf_{t_k})_{0 \le k \le n}-$ martingale. It is
clear from the proof proposed by~\cite{rogers} that in the dual price of a Bermudan option
(see~\eqref{eq:prix_dual}) the maximum can be shrunk to the random interval $[\tau_0, n]$.
Hence, it is sufficient to consider 
\begin{equation*}
  \inf_{\lambda \in \R^{\spaceIndex}, \; \lambda_0 = 0}
  \E\left[ \max_{\tau_0 \le k \le n} (Z_{t_k} - M_k(\lambda)) \right].
\end{equation*}
Using Doob's stopping theorem, we have, for any fixed $\lambda$, 
\[
  \E\left[ \max_{\tau_0 \le k \le n} (Z_{ t_k} - M_{k}(\lambda))\right]
  = \E\left[ \max_{\tau_0 \le k \le n} (Z_{ t_k} - (M_{k}(\lambda) -
  M_{\tau_0}(\lambda)))\right]
  = \E\left[ \max_{\tau_0 \le k \le n} (Z_{ t_k} - N_{k}(\lambda))\right].
\]
We deduce from this equality that minimizing over either set of martingales $M(\lambda)$
or $N(\lambda)$ leads to the same minimum value and that both problems share the same
properties, which justifies why we did not take into account the in--the--money condition
for the theoretical study. However, considering the set of martingales $N^\lambda$ is far
more efficient from a practical point of view. 

In our numerical examples, we modify $V_{p,n}$ and $V_{p,n}^m$ to take into account this
improvement and consider instead
\begin{align*}
  \tilde V_{p,n}(\lambda)  = \E\left[ \max_{\tau_0 \le k \le n} (Z_{t_k} - N_k(\lambda)) \right]
  \quad \text{and} \quad
  \tilde V^m_{p,n}(\lambda)  = \frac{1}{m} \sum_{i=1}^m \max_{\tau_0 \le k \le n} (Z^{(i)}_{t_k} -
  N^{(i)}_k(\lambda)).
\end{align*}
The idea of using martingales starting from the first time the option goes in the money is
actually owed to \cite{rogers}. Although he did not discuss it much, this was his choice in
the examples he treated.

\subsection{Our implementation of the algorithm}

To practically compute the infimum of $\tilde V^m_{p,n}$, we advise to use a gradient
descent algorithm, see Algorithm~\ref{algo}. The efficiency of such an approach mainly
depends on the computation of the descend direction. When the problem is not twice
differentiable, the gradient at the current point is used as a descent direction but it often
needs to be scaled, which makes the choice of the step size $\alpha_\l$ a burning issue to
ensure a fast numerical convergence. We refer to~\cite{boyd03} for a comprehensive survey
of several step size rules. After many tests, we found that the step size rule
proposed by~\cite{polyak87} was the best performing one in our context
\begin{align*}
  \alpha_\l = \frac{\tilde V^m_{p,n}(x_\l) - v^\sharp}{\norm{\nabla \tilde V^m_{p,n}(x_\l)}^2}
\end{align*}
where $v^\sharp$ is the price of the American option we are looking for. In practice, we
use the price of the associated European option instead of $v^\sharp$, which makes
$\alpha_\l$ too large and explains the need of the magnitude factor $\gamma$. The value of
the European price does not need to be very accurate. A decent and fast approximation can be
computed with a few thousand samples within few seconds no matter the dimension of the
problem.

\begin{algorithm}[ht]
    Generate $(G^{(1)}, Z^{(1)}), \dots, (G^{(m)}, Z^{(m)})$ $m$ i.i.d. samples
    following the law of $(Z, G)$ \;
    $x_0 \gets 0 \in \R^{\spaceIndex}$\;
    $\l \gets 0$, $\gamma \gets 1$, $d_0 \gets 0$, $v_0 \gets \infty$ \;
    \label{a:outer} \While{True}{
      Compute $v_{\l+1/2} \gets \tilde V_{p,n}^m(x_\l - \gamma \alpha_\l d_\l) $ \;
       \label{a:compute_V} \uIf{$v_{\l+1/2} < v_\l$}{
        $x_{\l+1} \gets x_\l - \gamma \alpha_\l d_\l$ \;
        $v_{\l+1} \gets v_{\l+1/2}$ \;
        $d_{\l+1} \gets \nabla \tilde V_{p,n}^m(x_{\l+1})$ \;\label{a:dk1} 
        \lIf{$\frac{\abs{v_{\l+1} - v_\l}}{v_\l} \le \varepsilon$}{\Return}
      } 
      \Else {
        $\gamma \gets \gamma / 2$ \;
      }
    }
  \caption{Sample Average Approximation of the dual price}
  \label{algo}
\end{algorithm}

To better understand how this algorithm works, it is important to note that as
$N(\lambda)$ linearly depends on $\lambda$, $N(\lambda) = \lambda \cdot \nabla_\lambda
N(\lambda)$ and therefore both the value function and its gradient are computed at the
same time without extra cost.  So, $\nabla \tilde V_{p,n}^m(x_{\l+1})$ is not actually
computed on line~\ref{a:dk1}  but at the same time as $v_{\l+1/2}$ on
line~\ref{a:compute_V}.

\SetKwBlock{Parallel}{In parallel do}{end}
\SetKw{Broadcast}{Broadcast}
\SetKw{Reduce}{Make a reduction of}
\begin{algorithm}[ht]
  \Parallel{Generate $(G^{(1)}, Z^{(1)}), \dots, (G^{(m)}, Z^{(m)})$ $m$ i.i.d. samples
    following the law of $(Z, G)$}
    $x_0 \gets 0 \in \R^{\spaceIndex}$ \;
    $\l \gets 0$, $\gamma \gets 1$, $d_0 \gets 0$, $v_0 \gets \infty$ \;
    \While{True}{
      \Broadcast{$x_\l$, $d_\l$, $\gamma$, $\alpha_\l$}\;
      \Parallel{Compute $\max_{\tau_0 \le k \le n} (Z^{(i)}_{t_k} - N^{(i)}_k(x_\l - \gamma
      \alpha_\l d_\l))$ for $i=1,\dots,m$}
      \Reduce{the above contributions to obtain  $\tilde V_{p,n}^m(x_\l - \gamma \alpha_\l d_\l)$ and
      $\nabla \tilde V_{p,n}^m(x_\l - \gamma \alpha_\l d_\l)$}\;
      $v_{\l+1/2} \gets \tilde V_{p,n}^m(x_\l - \gamma \alpha_\l d_\l)$ \;
       \uIf{$v_{\l+1/2} < v_\l$}{
        $x_{\l+1} \gets x_\l - \gamma \alpha_\l d_\l$ \;
        $v_{\l+1} \gets v_{\l+1/2}$ \;
        $d_{\l+1} \gets \nabla \tilde V_{p,n}^m(x_{\l+1})$ \; 
        \lIf{$\frac{\abs{v_{\l+1} - v_\l}}{v_\l} \le \varepsilon$}{\Return}
      } 
      \Else {
        $\gamma \gets \gamma / 2$ \;
      }
    }
    \caption{Parallel implementation of the Sample Average Approximation of the dual price}
    \label{algo-parallel}
\end{algorithm}

\paragraph{The HPC approach.} Our method targets large problems with as many as several
thousands of components for $\lambda$. This requires to design a scalable algorithm capable
of making the most of cluster architectures with hundreds of nodes. At each iteration, the
computation of $\tilde V_{p,n}^m$ and $\nabla \tilde V_{p,n}^m$ is nothing but a standard
Monte Carlo method and it inherits from its embarrassingly parallel nature.

A parallel algorithm for distributed memory systems based on the master/slave paradigm is
proposed in Algorithm~\ref{algo-parallel}. At the beginning, each process samples a bunch
of the $m$ paths (lines 1--3). Then, at each iteration the master process broadcasts the
value of $d_\l$, $x_\l$, $\alpha_\l$ and $\gamma$ (line 7 of Algorithm~\ref{algo}). With
these new values, each process computes its contribution to $\tilde V_{p,n}^m(x_\l -
\gamma \alpha_\l d_\l)$ and $\nabla \tilde V_{p,n}^m(x_\l - \gamma \alpha_\l d_\l)$ (lines
8--9) and the Monte Carlo summations are obtained by two simple reductions (line 11).
Then, the master process tests whether the move is admissible and updates the parameter
for the next iteration or returns the solution if the algorithm is not moving enough
anymore. This part carried out by the master process is very fast compared to the rest of
the code and we dare say that there is no centralized computation in our algorithm.
Moreover the communications are reduced to fours broadcasts, which guarantees an almost
perfect very good scalability. The number of communications is monitored by the number of
function evaluations, which remains quite small (between $10$ and $20$). We study the
efficiency of our algorithm on a few examples at the end of Section~\ref{sec:numerics}. 

\paragraph{Study of the complexity.} Most of the computational time is spent computing the
martingale part; remember that the cardinality of $\cc_{p,n}$ is given by $\binom{n d +
p}{n d} = \frac{(n d + p) \dots (n d + 1)}{p!}$. Using martingales only starting once the
option has been in the money enables us to only compute the martingale part on paths going
in the money strictly before maturity time. Depending on the product, this may allow for 
saving a lot of computational time. The complexity of one iteration of loop
line~\ref{a:outer} in Algorithm~\ref{algo} is proportional to 
\begin{align*}
  \sharp\{\text{paths in the money strictly before T}\} \times \binom{n d + p}{n d}.
\end{align*}
The payoffs are computed once and for all before starting the descent algorithm. It is
worth noting that its computational cost becomes negligible compared to the optimization
part when the dimension of the model or the number of dates increase, the most demanding
computation being the evaluation of the martingale decomposition.

\section{Applications}
\label{sec:numerics}

\subsection{Some frameworks satisfying the assumption of Proposition
  \ref{prop:differentiability}}
\label{sec:density_condition}

Let $(r_t)_t$ be the instantaneous interest rate supposed to be deterministic. 

\subsubsection{A put basket option in the multi--dimensional Black Scholes model}
\label{sec:basket-bs}

The $d-$dimensional Black Scholes model writes fori $j \in \{1, \dots, d\}$
\begin{align*}
  dS^j_t = S^j_t ( (r_t - \delta^j) dt + \sigma^j L_j dB_t)
\end{align*}
where $W$ is a Brownian motion with values in $\R^d$, $\sigma_t = (\sigma_t^1, \dots,
\sigma_t^d)$ is the vector of volatilities, assumed to be deterministic and positive at
all times, $\delta = (\delta^1, \dots, \delta^d)$ is the vector of instantaneous dividend
rates and $L_j$ is the $j$-th row of the matrix $L$ defined as a square root of the
correlation matrix $\Gamma$, ie.  $\Gamma = L L'$. Moreover, we assume that $L$ is lower
triangular. Clearly, for every $t$, the random vector $S_t$ is an element of $\D^{1,2}$.

The payoff of the put basket option writes as $\phi(S_t) = \left( K - \sum_{i=1}^d
  \omega^j S_t^j \right)_+$ where $\omega = (\omega^1, \dots, \omega^d)$ is a vector of real
valued weights. The function $\phi$ is Lipschitz continuous and hence $\phi(S_t) \in
\D^{1,2}$ for all $t$. Moreover, for $s \le t$ and $q \in \{1, \dots, d\}$, we have on the set
$\{\phi(S_t) > 0\}$
\begin{align*}
  D^q_s \phi(S_t) = \sum_{j=1}^d \omega^j S_t^j \sigma^j L_{j,q}.
\end{align*}
In particular for $q=d$, we get $D^d_s \phi(S_t) = \omega^d S_t^d \sigma^d L_{d,d}$.

Let $1 \le k \le n$ and $F$ be a non zero and $\cf_{t_k}-$measurable element of
$\cc_{p-1,n}$, ie.
\[
  F = \sum_{\alpha \in A^{\otimes d, k}_{p-1, n}} \lambda_\alpha
  \hh^{\otimes d}_{\alpha} \left( G_1, \dots, G_n \right)
\]
for some $\lambda \in \R^{A^{\otimes d}_{p, n}}$. 
Let $1 \le r \le k$.
\begin{align}
  \label{eq:Dt_bs_basket}
  &  \P\left(\forall t \in ]t_{r-1},t_r], \; D_t^d \phi(S_{t_k}) + F  = 0 \cond
    \phi(S_{t_k}) > 0\right) \nonumber \\
  &  = \P\left(\forall t \in ]t_{r-1},t_r], \; \omega^d S_{t_k}^d \sigma_t^d L_{d,d} + F  = 0
    \cond \phi(S_{t_k}) > 0 \right) \nonumber \\
  &  \le \frac{\P\left(\forall t \in ]t_{r-1},t_r], \; \omega^d S_{t_k}^d \sigma_t^d
      L_{d,d} + F  = 0 \right)}{\P(\phi(S_{t_k}) > 0)}.
\end{align}
If $p=1$, then $F$ is a deterministic non zero constant. In this case, the numerator
vanishes because $S_{t_k}^d$ has a density. Assume $p \ge 2$, then $F$ is a multivariate
polynomial with global degree $p - 1 \ge 1$. Then we can find $\l \in \{1, \dots, k\}$,
$q \in \{1, \dots, d\}$ and $\alpha$ such that $\alpha_\l^q \ge 1$ and $\lambda_\alpha \ne
0$. Let $\hat\cg$ be the sigma algebra generated by $(G_i^j, 1 \le i \le k, 1 \le
j \le d, (i, j) \ne (\l,q))$.
\begin{align*}
  \P\left(\forall t \in ]t_{r-1},t_r], \; \omega^d S_t^d \sigma_t^d L_{d,d} + F  = 0 \right) =
  \E\left[ 
    \P\left(\forall t \in ]t_{r-1},t_r], \; \omega^d S_t^d \sigma_t^d L_{d,d} + F  = 0 
  \cond \hat \cg \right) \right].
\end{align*}
Conditioning on $\hat\cg$, the random variable $\omega^d S_t^d \sigma_t^d L_{d,d} + F$ only
depends on $G_\l^q$. Consider the algebraic equation for $x \in \R$
\begin{equation}
  \label{eq:algebraic}
  a \expp{b x + c} = P(x)
\end{equation}
where $(a, b, c) \in \R^3, a \ne 0, b \ne 0$ and $P$ is polynomial with degree $p-1 \ge
1$. Let $f(x) = a \expp{b x + c} - P(x)$, $f^{(p)}(x) = a b^p \expp{b x + c}$. Clearly,
$f^{(p)}$ never vanishes, which ensures that $f$ has at most $p$ different roots. Hence,
we deduce that for any $t \in ]t_{r-1},t_r]$,  $\P\left(\omega^d S_t^d \sigma_t^d L_{d,d} +
F  = 0 \cond  \hat \cg \right) = 0 $. Combining this result along with~\eqref{eq:Dt_bs_basket}
proves that Equation~\eqref{eq:Dt_density} holds in this setting.

\subsubsection{A put option on the minimum of a basket in the multi--dimensional Black Scholes model} 
\label{sec:min-bs}
  
We use the notation of the previous example. The payoff of the put option on the minimum
of $d$ assets write $\phi(S_t) = (K - \min_j(S^j_t))_+$. One can prove by induction on $d$
that the function $x \in \R^d \longmapsto \min_j(x^j)$ is $1-$Lipschitz for the $1-$norm
on $\R^d$. Hence, as the positive part function is also Lipschitz, the payoff function
$\phi$ is Lipschitz. Then, \cite[Proposition 1.2.4]{nualart_98} yields that for all $t \in
[0,T]$, $\phi(S_t) \in \D^{1,2}$ and for all $q \in \{1,\dots, d\}$, 
\[
  D^q(\phi(S_t)) = \sum_{j=1}^d \partial_{x^j} \phi(S_t) D^q(S_t^j) = \sum_{j=1}^d
  \partial_{x^j} \phi(S_t) S_t^j \sigma^j L_{j,q}.
\]
With our choice for the matrix $L$, 
\[
  D^d(\phi(S_t)) = \partial_{x^d} \phi(S_t)  S_t^d \sigma^d L_{d,d}
  = -  S_t^d \sigma^d L_{d,d} \ind{\phi(S_t) > 0} \ind{\min_j (S_t^j) = S_t^d}.
\]
Let $1 \le k \le n$ and $F$ be a non zero and $\cf_{t_k}-$measurable element of
$\cc_{p-1,n}$. For $1 \le r \le k$,
\begin{align*}
  &  \P\left(\forall t \in ]t_{r-1},t_r], \; D_t^d \phi(S_{t_k}) + F  = 0 \cond
    \phi(S_{t_k}) > 0\right) \nonumber \\
  & =  \P\left(\forall t \in ]t_{r-1},t_r], \; -  S_t^d \sigma^d L_{d,d} + F  = 0 \cond
  \phi(S_{t_k}) > 0, \; \min_j (S_t^j) = S_t^d\right) \P\left(\min_j (S_t^j) = S_t^d\right) \nonumber \\
  & \quad +  \P\left(\forall t \in ]t_{r-1},t_r], \; F  = 0 \cond
  \phi(S_{t_k}) > 0, \; \min_j (S_t^j) \ne S_t^d\right) \P\left(\min_j (S_t^j) \ne S_t^d\right)
\end{align*}
Clearly, the second term in the above sum is zero as $F$ has a density. Hence,
\begin{align*}
  \P\left(\forall t \in ]t_{r-1},t_r], \; D_t^d \phi(S_{t_k}) + F  = 0 \cond
    \phi(S_{t_k}) > 0\right)
   \le  \frac{\P\left(\forall t \in ]t_{r-1},t_r], \; -  S_t^d \sigma^d L_{d,d} + F  =
    0\right)}{ \P(\phi(S_{t_k}) > 0)}.
\end{align*}
We conclude as in the case of the put basket option.

\subsubsection{A put option in the Heston model} 
\label{sec:put-hes}

The Heston model can be written
\begin{align*}
  dS_t &= S_t(r_t dt + \sqrt{\sigma_t} (\rho dW^1_t + \sqrt{1 - \rho^2} dW^2_t) \\
  d\sigma_t &=  \kappa (\theta - \sigma_t) dt + \xi \sqrt{\sigma_t} dW^1_t.
\end{align*}
For $s \le t$, $D^2_s S_t = S_t \sqrt{1-\rho^2} \sqrt{\sigma_t}$. Conditionally on $W^1$,
$D^2_s S_t$ writes as $a \expp{b W^2_t + c}$ and we can unfold the same reasoning
as after~\eqref{eq:algebraic}.

\subsection{Numerical experiments}
In this part, we present results obtained from a sequential implementation of our approach
as described in Algorithm~\ref{algo}. The computations are run on a standard laptop with
an Intel Core i5 processor 2.9 Ghz. For each experiment, we report the price obtained
using Algorithm~\ref{algo} along with its computational time and standard deviation.

\subsubsection{Examples in the Black Scholes models}

We consider the $d-$dimensional Black Scholes as presented in Section~\ref{sec:basket-bs}.
For the sake of simplicity in choosing the parameters, we have decide to use the same
correlation between all the assets, which amounts to considering the following simple
structure for $\Gamma$.
\begin{equation}
  \label{eq:covstruct}
  \Gamma = \begin{pmatrix}
    1 & \rho & \hdots & \rho\\
    \rho & 1 &\ddots & \vdots\\
    \vdots&\ddots&\ddots& \rho\\
    \rho &\hdots & \rho & 1 
  \end{pmatrix}
\end{equation}
where $\rho \in ]-1 / (d-1), 1]$ to ensure that $\Gamma$ is positive definite.

\paragraph{A basket option in the Black--Scholes model.}

We consider a put option on several assets as presented in Section~\ref{sec:basket-bs}.
We report in Table~\ref{tab:basket-bs} the price obtained with our approach for $m=20,000$.
The last column \emph{reference price} corresponds to the prices reported
in~\cite{schoen12-1} on the same examples. These \emph{reference} prices were obtained
within a few minutes according to the authors whereas here we manage to get similar values
within a few seconds. We can see that a second order chaos expansion, $p=2$, already
gives very accurate results within a few tenths of a second for a $5-$dimensional problem
with $6$ dates, which proves the impressive efficiency of our approach.
\begin{table}[ht]
  \centering \begin{tabular}{ccccccc}
    $p$ &  $n$ & $S_0$ &  price & Stdev & time (sec.) & reference price \\
    \hline
    $2$ & $3$ & $100$ & $2.27$ & $0.029$ & $0.17$  & $2.17$ \\
    $3$ & $3$ & $100$ & $2.23$ & $0.025$ & $0.9$ & $2.17$ \\
    $2$ & $3$ & $110$ & $0.56$ & $0.014$ & $0.07$  & $0.55$ \\
    $3$ & $3$ & $110$ & $0.53$ & $0.012$ & $0.048$ & $0.55$ \\
    $2$ & $6$ & $100$ & $2.62$ & $0.021$ & $0.91$  & $2.43$ \\
    $3$ & $6$ & $100$ & $2.42$ & $0.021$ & $14$    & $2.43$ \\
    $2$ & $6$ & $110$ & $0.61$ & $0.012$ & $0.33$  & $0.61$ \\
    $3$ & $6$ & $110$ & $0.55$ & $0.008$ & $10$    & $0.61$ \\
    \hline
  \end{tabular}
  \caption{Prices for the put basket option with parameters $T=3$, $r = 0.05$, $K=100$, $\rho=0$,
  $\sigma^j = 0.2$, $\delta^j = 0$, $d = 5$, $\omega^j = 1 / d$.}
  \label{tab:basket-bs}
\end{table}

\paragraph{A call on the maximum of $d$ assets in the Black--Scholes model.}

We consider a call option on the maximum of $d$ assets in the Black Scholes model.  As in
the previous example, the last column \emph{reference price} corresponds to the prices reported
in~\cite{schoen12-1} on the same examples.
\begin{table}[ht]
  \centering \begin{tabular}{cccccccc}
    $d$ & $p$ & $m$ & $S_0$ & price & Stdev & time (sec.) & reference price \\
    \hline
    $2$ & $2$ & $20,000$ & $90$  & $10.18$ & $0.07$ & $0.4$  & $8.15$  \\
    $2$ & $3$ & $20,000$ & $90$  & $8.5$   & $0.05$ & $4.1$  & $8.15$  \\
    $2$ & $2$ & $20,000$ & $100$ & $16.2$  & $0.06$ & $0.54$ & $14.01$ \\
    $2$ & $3$ & $20,000$ & $100$ & $14.4$  & $0.06$ & $5.6$  & $14.01$ \\
    $5$ & $2$ & $20,000$ & $90$  & $21.2$  & $0.09$ & $2$    & $16.77$ \\
    $5$ & $3$ & $40,000$ & $90$  & $16.3$  & $0.05$ & $210$  & $16.77$ \\
    $5$ & $2$ & $20,000$ & $100$ & $30.7$  & $0.09$ & $3.4$  & $26.34$ \\
    $5$ & $3$ & $40,000$ & $100$ & $26.0$  & $0.05$ & $207$  & $26.34$ \\
    \hline
  \end{tabular}
  \label{tab:max-bs}
  \caption{Prices for the call option on the maximum of $d$ assets with
    parameters $T=3$, $r = 0.05$, $K=100$, $\rho=0$, $\sigma^j = 0.2$, $\delta^j = 0.1$, $n =
  9$.}
\end{table}
 With no surprise, the computational time increases exponentially with the dimension $n
 \times d$ and the degree $p$. Whereas a second order expansion provides very
 accurate results for the basket option, it only gives a rough upper--bound for the call
 option on the maximum of $d$ assets. Considering a third order expansion $p=3$ takes far
 longer but enables us to get very tight upper--bounds.

\paragraph{A geometric basket option in the Black--Scholes model}

Benchmarking a new method on high dimensional products becomes hardly feasible as almost no
high dimensional American options can be priced accurately in a reasonable time. An
exception to this is the geometric option with payoff $(K - (\prod_{j=1}^d
S^j_t)^{1/d})_+$ for the put option. Easy calculations show that the price of
this $d-$dimensional option equals the one of the $1-$dimensional option with parameters
\begin{align*}
  \hat S_0 = \left( \prod_{j=1}^d S_0^j \right)^{1/d}; \quad
  \hat \sigma = \inv{d} \sqrt{\sum_{i,j} \sigma^i \sigma^j \Gamma_{ij}}; \quad
  \hat \delta  = \frac{1}{d} \sum_{j=1}^d \left( \delta^j + \inv{2} (\sigma^j)^2 \right)
  - \inv{2} (\hat \sigma)^2.
\end{align*}
Table~\ref{tab:geom-correspondence} summarizes the correspondence values used in the
examples.
\begin{table}[ht]
  \centering
  \begin{tabular}{ccccccc}
    $d$  & $S_0$ & $\sigma$ & $\rho$ & $\hat S_0$ & $\hat \sigma$ & $\hat \delta$ \\
    \hline
    $2$  & $100$ & $0.2$    & $0$    & $100$      & $0.14$        & $0.01$      \\
    $10$ & $100$ & $0.3$    & $0.1$  & $100$      & $0.131$       & $0.036$       \\
    $40$ & $100$ & $0.3$    & $0.1$  & $100$      & $0.105$       & $0.039$ \\
    \hline
  \end{tabular}
  \caption{Correspondence table for the parameters of the geometric options with $\delta^j = 0$.}
  \label{tab:geom-correspondence}
\end{table}

\begin{table}[ht]
  \centering
  \begin{tabular}{ccccccccc}
    $d$  & $ \sigma^j$ & $ \rho$ & $p$ & $m$     & price  & Stdev  & time(sec) & $1-$d price \\
    \hline
    $2$  & $0.2$ & $0$   & $2$ & $5000$  & $4.32$ & $0.04$ & $ 0.018$ & $4.20$ \\
    $2$  & $0.2$ & $0$   & $3$ & $5000$  & $4.15$ & $0.04$ & $1.3$    & $4.20$ \\
    $10$ & $0.3$ & $0.1$ & $1$ & $5000$   & $5.50$ & $0.06$ & $0.12$   & $4.60$ \\
    $10$ & $0.3$ & $0.1$ & $2$ & $20000$  & $4.55$ & $0.02$ & $17$     & $4.60$ \\
    $40$ & $0.3$ & $0.1$ & $1$ & $10000$ & $4.4$  & $0.03$ & $1.4$    & $3.69$ \\
    $40$ & $0.3$ & $0.1$ & $2$ & $20000$ & $3.61$ & $0.02$ & $170$    & $3.69$ \\
    \hline
  \end{tabular}
  \caption{Prices for the geometric basket put option with parameters $T=1$, $r = 0.0488$
  (it corresponds to a $5\%$ annual interest rate), $K=100$, $\delta^j = 0$, $n = 9$.}
  \label{tab:geom-bs}
\end{table}

The $1-d$ price is computed using  a tree method with several thousand steps.
We can see in Table~\ref{tab:geom-bs} that a second order approximation gives very
accurate result within a few seconds for an option with $10$ underlying assets, which
proves the efficiency of our approach. We cannot beat the curse of dimensionality, which
slows down of algorithm for very large problems. For an option on $40$ assets, we obtain
a price up to a $3\%$ relative error within $3$ minutes which is already very
fast for such a high dimensional problem. The number of terms involved in the chaos
expansion can become very large: for $d=40$ and $p=2$, there are $65340$ elements in
$\cc_{p,n}$. Even though we are not working in a linear algebra framework, it is advisable
to ensure that the number of samples $m$ used in the sample average approximation is
larger than the number of free parameters in the optimization problem. When $m$ becomes
too small, we may face an over--fitting phenomenon as the number of parameters is far too
large compared to the information contained in the sample average approximation. This
probably explains why the price obtained for $p=2$, $d=40$ and $m=40$ is slightly smaller
than the true price.

In the next paragraph, we test the scalability of Algorithm~\ref{algo-parallel} on this
particular examples for a larger number of samples.

\subsubsection{Scalability of the parallel algorithm}

We consider the $40-$dimensional geometric put option studied in Table~\ref{tab:geom-bs}
with $p=2$ and test the scalability of our parallel implementation for $m=200,000$. The
tests are run on a BullX DLC supercomputer containing $190$ nodes for a total of $3204$
CPU cores. We report in Table~\ref{tab:scalability} the results of our scalability study
using from $1$ to $512$ cores. Despite the two levels of parallelism available on this
supercomputer, we have used a pure MPI implementation without any reference to multithread
programming. We could probably have improved the efficiency a bit using two levels of
parallelism, but the results are already convincing enough and do not justify the need of
a two level approach, which makes the implementation more delicate.  The sequential
Algorithm runs within one hour and a quarter whereas using $512$ cores we manage to get
the computational time down to a dozen of seconds, which corresponds to a $0.6$
efficiency. Considering the so short wall time required by the run on $512$ cores, keeping
the efficiency at this level represents a great achievement. Note that with $128$ cores,
the code runs within a minute with an efficiency of three quarters. These experiments
prove the impressive scalability of our algorithm.

\begin{table}[ht]
  \centering
  \begin{tabular}{ccc}
    \#processes & time (sec.) & efficiency \\
    \hline
    1           & 4365        & 1          \\
    2           & 2481        & 0.99       \\
    4           & 1362        & 0.90       \\
    16          & 282         & 0.84       \\
    32          & 272         & 0.75       \\
    64          & 87          & 0.78       \\
    128         & 52          & 0.73       \\
    256         & 34          & 0.69       \\
    512         & 10.7        & 0.59       \\
    \hline 
  \end{tabular}
  \caption{Scalability of Algorithm~\ref{algo-parallel} on the $40-$dimensional
  geometric put option described above with $T=1$, $r = 0.0488$, $K=100$, $\sigma^j=0.3$,
$\rho=0.1$, $\delta^j = 0$, $n = 9$, $p=2$.}
  \label{tab:scalability}
\end{table}

\section{Conclusion}

We have proposed a purely dual algorithm to compute the price of American or Bermudan
options using some stochastic optimization tools. The starting point of our algorithm is
the use of Wiener chaos expansion to build a finite dimensional vector space of
martingales.  Then, we rely on a sample average approximation to effectively optimize the
coefficients of the expansion.  Our algorithm is very fast: for problems up to
dimension $5$, a price is obtained within a few seconds, which is a tremendous improvement
compared to existing purely dual methods. For higher dimensional problems, we can use a
very scalable parallel algorithm to tackle very high dimensional problems ($40$ underlying
assets). We can transparently deal with complex path--dependent payoffs without any extra
computational cost. Event though, we restricted to a Brownian setting in this work, our
approach could easily be extended to jump diffusion models by introducing Poisson chaos
expansion, which is linked to Charlier polynomials (see~\cite{Geiss2016}). We believe that
our approach could be improved by cleverly reducing the number of terms in the chaos
expansion, the computation of which centralizes most of the effort.

\bibliographystyle{abbrvnat}
\bibliography{chaos-am-biblio.bib}
\end{document}